\documentclass{amsart}

\usepackage[a4paper,margin=3.78cm]{geometry}
\usepackage{amsthm,amssymb,mathtools,mathrsfs}
\usepackage{comment,color,lineno}
\usepackage{xspace}
\usepackage{url}

\theoremstyle{plain}
\newtheorem{lemma}{Lemma}[section]
\newtheorem{corollary}[lemma]{Corollary}
\newtheorem{proposition}[lemma]{Proposition}
\newtheorem{theorem}[lemma]{Theorem}
\newtheorem{conjecture}[lemma]{Conjecture}
\newtheorem{question}[lemma]{Question}

\theoremstyle{definition}
\newtheorem{definitioN}[lemma]{Definition}

\theoremstyle{remark}

\newtheorem{remark}[lemma]{Remark}

\newcommand{\Z}{\mathbb{Z}}

\newcommand{\N}{\mathbb{N}}

\title{Snakes can be fooled into thinking they live in a tree}

\author{Laurent Bartholdi \and Ville Salo}

\date{September 17, 2024}

\dedicatory{Happy birthday, and long life, Jarkko!}

\begin{document}
\maketitle

\begin{abstract}
We construct a finitely generated group which is not virtually free, yet has decidable snake tiling problem. This shows that either a long-standing conjecture by Ballier and Stein (the characterization of groups with decidable domino problem as those virtually free ones) is false, or a question by Aubrun and Bitar has a positive answer (there exists a group for which the domino and snake problems are of different difficulty).
\end{abstract}

\section{Introduction}
The \emph{domino problem} on a finitely generated group $G=\langle S\rangle$ asks to decide, given finite sets $C$ of ``colours'' and $D\subseteq C\times S\times C$ of ``dominoes'', whether $G$'s Cayley graph can be vertex-coloured by $C$ in such a manner that every edge carries a domino in $D$.

When $G=\Z$ this is the classical problem with standard $2\times 1$ dominoes, and is easily solvable. When $G=\Z^2$ dominoes are usually called \emph{Wang tiles}, and Berger's celebrated result~\cite{Berger66} is that this problem is unsolvable.

The domino problem belongs in fact to a very small fragment of the monadic second-order logic of $G$'s Cayley graph. Much attention has been devoted to delineating the precise boundary between decidability and undecidability, and the Ballier-Stein conjecture states that it is decidable precisely when the whole monadic second-order logic is decidable, namely precisely when $G$ has a finite-index free subgroup (it is \emph{virtually free}):

\begin{conjecture}[Ballier \& Stein~\cite{BS18}]
\label{con:pgdomino}
  The domino problem on a finitely generated group $G$ is decidable if and only if $G$ is virtually free.
\end{conjecture}
Recall that a group is virtually free if and only if its Cayley graph is quasi-isometric (see~\S\ref{ss:amalgams}) to a tree.

\emph{Snake tiling problems} relax the constraint that $G$'s Cayley graph be entirely vertex-coloured; the basic variant requires merely that an infinite path (the \emph{snake}) be vertex-coloured, with legal dominoes on the path's edges. Snake problems were introduced by Myers~\cite{My79}, and this specific one appears in~\cite{Et91,EtHaMy94}. Kari~\cite{Kari03} proved (as conjectured in~\cite{EtHaMy94}) that it is unsolvable for the grid $\Z^2$. Kari's solution arises in fact from his independent earlier work in~\cite{Ka94} on undecidability of the reversibility (and surjectivity) problem for cellular automata, where he used a variant of the snake problem for directed tile sets (where the snake can in a sense choose where it moves).

Aubrun and Bitar~\cite{AB23} consider numerous variants of the snake problem on finitely generated groups, and ask:

\begin{question}[Aubrun \& Bitar~{\cite[Question~8.2]{AB23}}]
Is there a finitely generated group with undecidable Domino Problem and decidable snake problem? Is there such a group where the inverse holds?
\end{question}

They show that essentially all variants of the snake tiling problem are decidable on virtually free groups, and conversely are undecidable on a large class of groups including all (central infinite non-torsion)-by-(infinite non-torsion) groups. Note that the domino problem is independent of the choice of generating set $S$, while this is unclear for the snake problems.

These results might lead credence to the conjecture that snake problems are decidable precisely on virtually free groups; but we show in this article that such a variant of the Ballier-Stein conjecture is not true:

\begin{theorem}
  There exists a finitely generated group which is not virtually free, but which has decidable snake tiling problem for any generating set.
\end{theorem}

In particular, either the Ballier-Stein conjecture is false (if the domino problem is solvable for this group), or the Aubrun-Bitar question has a positive answer for this group. We conjecture in fact that the former holds, and that an improvement of our construction will lead to a non-virtually free group with solvable domino problem.

\def\mathvisiblespace{\text{\textvisiblespace}}

\section{Snake problems}
We collect here some basic definitions, and an overview of the different variants of the snake problem.

Recall that, for a finitely generated group $G=\langle S\rangle$ with $S$ a finite generating set, its \emph{Cayley graph} is the graph with vertex set $G$ and for each $g\in G,s\in S$ an edge labeled `$s$' from $g$ to $g s$.

\begin{definition}[Tileset]
  Let $G=\langle S\rangle$ be a finitely generated group. A \emph{tileset} for $G$ is a pair $\Theta=(C,D)$ of finite sets with $D\subseteq C\times S\times C$.

  If the dependency on $S$ is unclear, we will denote the tileset by $\Theta=(S,C,D)$, and call $S$ the \emph{memory set} of $\Theta$.
\end{definition}

We may visualize a tileset as a multigraph still written $\Theta$ with vertex set $C$, and for each $(c,s,c')\in D$ an edge from $c$ to $c'$ labeled `$s$'.

\begin{definition}[Domino problem]
  The \emph{domino problem} for a finitely generated group $G=\langle S\rangle$ asks to determine, given a tileset $(C,D)$, whether there exists a vertex-colouring $x\colon G\to C$ of $G$'s Cayley graph such that $(x(g),s,x(g s))\in D$ for all $g\in G,s\in S$.
\end{definition}
In other words, the domino problem asks to determine whether there exists a map of labeled graphs from $G$'s Cayley graph to the graph associated with $(C,D)$.

We collect now a variety of \emph{snake problems} for a group $G=\langle S\rangle$. We insist that these problems \emph{a priori} depend on the choice of generating set $S$.
\begin{definition}[Snake problem]
  Let $G=\langle S\rangle$ be a finitely generated group. An \emph{snake} is an injective map $\omega\colon\Z/n\Z\to G$, for some $n\ge0$, such that $\omega(i)^{-1}\omega(i+1)\in S$ for all $i\in\Z/n\Z$. It is an \emph{infinite snake} if $n=0$ and an \emph{ouroboros} if $n\ge1$.

  Let $(C,D)$ be a tileset. The \emph{weak snake problem} asks whether there exists a snake $\omega\colon\Z/n\Z\to G$ and a colouring $x\colon\omega(\Z/n\Z)\to C$ such that $(x(\omega(i)),\omega(i)^{-1}\omega(i+1),x(\omega(i+1)))\in D$ for all $i\in\Z/n\Z$; then $(\omega, x)$ is called a \emph{weak snake}. It subdivides into the \emph{weak infinite snake problem} and \emph{weak ouroboros problem} if we further require $n=0$ or $n\ge1$.
  
  The \emph{strong snake problem} asks whether there exists a snake $\omega\colon\Z/n\Z\to G$ and a colouring $x\colon\omega(\Z/n\Z)\to C$ such that $(x(g),s,x(g s))\in D$ whenever $g,gs\in\omega(\Z/n\Z)$; then $(\omega, x)$ is called a \emph{strong snake}. Again it subdivides into the \emph{strong infinite snake problem} and \emph{strong ouroboros problem} according to whether $n=0$ or $n\ge1$.

  The \emph{directed weak/strong snake problem} has the modified requirement that the colouring be of the form $x\colon\omega(\Z/n\Z)\to C \subseteq C'\times S$, and the snake always follows the direction in the $S$-component, namely $x(\omega(i))=(*,\omega(i)^{-1}\omega(i+1))$ for all $i\in\Z/n\Z$.
\end{definition}

In total, we have defined $12$ different snake problems, by combining the (weak/strong), (directed/undirected) and (snake/infinite snake/ouroboros) constraints. These problems are stated for a finitely-generated group with a fixed generating set. We also define versions without a generating set, i.e.\ for $X=$ any of these $12$ problems, the \emph{$X$ snake problem for the group $G$} is ``given a finite generating set $S\subseteq G$ and a tileset for $G=\langle S\rangle$, is there an $X$?''.

Naturally, if problem $X$ is decidable for a group $G$, then it is decidable for the group $G=\langle S\rangle$ with any fixed generating set. We give simple reductions between the above problems.

The weak infinite snake problem naturally combines a classical $1$-dimensional domino problem, that of colouring the integers according to domino rules, with an added geometric constraint that every domino carries an element of $S$, and the path traced by the corresponding edges in $G$'s Cayley graph is injective.

The strong infinite snake problem is perhaps the most natural especially in the context of Conjecture~\ref{con:pgdomino}, in that it may equivalently be formulated as follows: ``does there exist an infinite, connected subgraph of $G$'s Cayley graph and a valid $(C,D)$-tiling of this subgraph?''. Indeed every infinite snake traces an infinite connected subgraph, and every infinite connected subgraph contains an infinite line, which can be made bi-infinite by taking a limit of translates.

On the other hand, the strong directed snake problem is quite robust, as the following lemmas show.

\begin{lemma}\label{lem:ReduceToStrong}
  The strong, weak and directed weak problems all reduce to the directed strong problem.
\end{lemma}

\begin{proof}
  We will show `strong $\leq$ directed strong', and then `weak $\leq$ directed weak $\leq$ directed strong'.

  Consider first an instance $(C,D)$ of the strong problem. Replace $C$ by $C\times S$. A directed strong snake in the new tileset $(C\times S,B)$ is just a strong snake for the original. Thus, `strong $\leq$ directed strong'.

  The same argument gives `weak $\leq$ directed weak'.

  Finally, consider an instance $(C\times S,D)$ of the directed weak problem, with $D\subseteq C\times S\times C$, namely the dominoes in $D$ do not know the directions. We may add this information by replacing $C$ by $C'\coloneqq C\times S$, and set
  \[B'\coloneqq\{((c,s),s,(c',s')):(c,s,c')\in B,s'\in S\}.\]
  In this manner, the weak directed problem reduces to the variant of the weak problem in which the direction information is directly available to the dominoes.

  With this assumption in place, we can add to $B'$ all triples $((c, s), s'',(c', s'))$ with $s'' \neq s$. Any strong directed snake will remain a strong directed snake, since these new triples do not appear on $\omega(\Z/n\Z)$. However, now every weak directed snake is actually a strong snake as well, so the strong snakes for the new tileset are precisely the weak snakes for the old tileset.
\end{proof}

\begin{lemma}\label{lem:BiggerGeneratingSet}
If $(S, C, D)$ is an instance of the strong directed snake problem, and $T \supset S$ is any given finite set, then we can compute an instance $(T, C', D')$ of the strong directed snake problem which is equivalent to $(S, C, D)$.
\end{lemma}

\begin{proof}
By the construction in the previous proof, we may assume that the direction information is available to the snakes. Then, informally speaking, we can use the same tileset, except we check the constraints only in the directions in $S$. We omit the details.
\end{proof}

Let us repeat more precisely our main result:
\begin{theorem}\label{thm:main}
  There exists a finitely generated group $G$ which is not virtually free, but for which the strong directed infinite snake problem and the strong directed snake problem are decidable.
\end{theorem}

Lemmata~\ref{lem:ReduceToStrong} and~\ref{lem:BiggerGeneratingSet} then imply that the other variants of snake and infinite snake problems are decidable for $G$.

We are not able to settle the problem for ouroboroi; the following remains open.

\begin{question}
  Is there a finitely generated group which is not virtually free, and for which some / all of the ouroboros problems are decidable?
\end{question}

\section{Amalgamated free products and HNN extensions}\label{ss:amalgams}
For groups $A,B,C$ with imbeddings $A\hookleftarrow C\hookrightarrow B$, recall that their amalgamated free product $A*_C B$ is the universal (``freest'') group generated by $A\cup B$ in which both images of $C$ coincide. Write $A=C\cdot T_A$ and $B=C\cdot T_B$ by choosing transversals $T_A,T_B$ of $C$ in $A,B$ respectively; without loss of generality $1\in T_A$ and $1\in T_B$. There is then a normal form for $A*_C B$, see~\cite[Theorem~IV.2.6]{LS77}: every element can be uniquely represented as the product of an element of $C$ with a word alternating in $T_A\setminus\{1\}$ and $T_B\setminus\{1\}$.

Recall that a quasi-isometry between metric spaces $X,Y$ (such as graphs) is a map $f\colon X\to Y$ such that, for some constant $K$, we have $K^{-1} d(x_1,x_2)-K\le d(f(x_1),f(x_2))\le K d(x_1,x_2)+K$ and $d(f(X),y)\le K$ for all $x_1,x_2\in X,y\in Y$. \emph{Geometric} properties of groups are those that may be defined on the group's Cayley graph and are invariant under quasi-isometry.

If $A,B$ are finite, then $G\coloneqq A*_C B$ is virtually free, and its Cayley graph is quasi-isometric to a tree. Better, there is a natural tree on which $G$ acts with finite stabilizers, its \emph{Bass-Serre tree}. Its vertices are right cosets of $A,B$ respectively in $G$, its edges are right cosets of $C$, with endpoints given by coset inclusion, and the action is given by right multiplication.

Closely related is the \emph{HNN extension} $A*_C$. It is the universal group generated by $A\cup\{t\}$ in which $t$ conjugates both imbeddings of $C$. There is a natural map $A*_C\to\Z$ given by $t\mapsto1,A\mapsto 0$, with kernel $\cdots *_C A *_C A *_C A *_C \cdots$. Again there is a normal form for $A*_C$, see~\cite[Theorem~IV.2.1]{LS77}: choosing as before transversals $T_-\ni1,T_+\ni1$ of both imbeddings of $C$ in $A$, every $g\in A*_C$ can be uniquely represented as
\begin{equation}\label{eq:hnn}
  g=g_0 t^{\varepsilon_1}g_1\cdots t^{\varepsilon_n}g_n
\end{equation}
with $g_0\in A$, $\varepsilon_i\in\{\pm1\}$, $g_i\in T_{\varepsilon_i}$ and no consecutive $t^{\pm1}1t^{\mp1}$.

If again $A$ is finite, then $G\coloneqq A*_C$ is virtually free, its Cayley graph is quasi-isometric to a tree, and there is a natural tree on which $G$ acts with finite stabilizers, its \emph{Bass-Serre tree}. Its vertices are right cosets of $A$ in $G$, its edges are right cosets of $C$ and $C^t$, with the edge $C g$ connecting $A g$ to $A t^{-1}g$ and the edge $C^t g$ connecting $A g$ to $A t g$, and the action is given by right multiplication.

By classical results of Baumslag and Tretkoff~\cite{Ba63,Tr73,BaTr78}, the groups $A*_C B$ and $A*_C$ are residually finite as soon as $A,B$ are finite.

\subsection{A tower of HNN extensions}
Our basic construction is as follows. Start by $A_{-1}=1$ and $A_0$ any finite group. Then, inductively: given two imbeddings of $A_{n-2}$ in $A_{n-1}$, let $A_n$ be a finite quotient of $A_{n-1}*_{A_{n-2}} A_{n-1}$ in which both $A_{n-1}$ imbed, and proceed with these two imbeddings. The art, in the construction, is to choose the quotient $A_n$ appropriately.

Let $G_n$ denote the HNN extension $A_n*_{A_{n-1}}$. The inclusion in the first factor induces homomorphisms $A_{n-1}\hookrightarrow A_n$ compatible with the inclusions of both copies of $A_{n-2}$, and therefore induces homomorphisms $\pi_n\colon G_{n-1}\to G_n$. Let $G$ denote the colimit of $(G_0\to G_1\to\cdots)$.

\begin{lemma}\label{lem:generate}
  $G$ is generated by $A_0\cup\{t\}$.
\end{lemma}

\begin{proof}
  It suffices to prove that $\pi_n$ is onto for all $n\ge1$. Now $G_n$ is generated by $A_n\cup\{t\}$ and therefore by $A_{n-1}\cup A_{n-1}'\cup\{t\}$ where $A_{n-1}'$ is the other copy of $A_{n-1}$ in $A_n$; and these two copies are conjugate by $t$ so $G_n$ is generated by $A_{n-1}\cup\{t\}$, which are the generators of $G_{n-1}$.
\end{proof}

\begin{lemma}
  If the $A_n$ are not eventually constant (and therefore increase unboundedly, since they are subgroups of each other), then $G$ is not virtually free.
\end{lemma}
\begin{proof}
  If indeed $G$ were virtually free, it would have a bound on the cardinality of its finite subgroups.
\end{proof}

There are quite a few natural examples of sequences of groups $A_n$, with isomorphisms $\alpha_n\colon A_{n-1}\to A_{n-1}'$ between two subgroups of $A_n$:
\begin{itemize}
\item $A_n=(\Z/2)^n$ and $\alpha_n\colon (\Z/2)^{n-1}\times1\to1\times (\Z/2)^{n-1}$. The resulting group $G$ is the ``lamplighter group'' $\Z/2\wr\Z$;
\item $A_n=S_n$ and $\alpha_n\colon S_{\{1,\dots,n-1\}}\to S_{\{2,\dots,n\}}$. The resulting group $G$ is the group of permutations of $\Z$ that act as a translation away from the origin;
\item $A_n=GL(n,\mathbb F_p)$ and $\alpha_n\colon GL(n,\mathbb F_p)\times1\to1\times GL(n,\mathbb F_p)$. The resulting group $G$ is the group of invertible $\mathbb F_p$-linear maps of $\mathbb F_p[t,t^{-1}]$ that act as a power of $t$ on large enough powers of $t$.
\end{itemize}

\begin{lemma}\label{lem:locfinite}
  For every finite $I\subset\Z$ and every $r\in\N$, the following holds for all $n\in\N$ large enough. Consider the Cayley graph of $G_n$ with generating set $B_{G_n}(r)$, and its subgraph $\mathcal C_{r,I}$ spanned by all elements whose image in $\Z$ belongs to $I$. Then all connected components of $\mathcal C$ are finite, and of bounded cardinality.
\end{lemma}
\begin{proof}
  We may replace $I$ by the interval $I'\coloneqq[\min I-r,\max I+r]$ and consider the subgraph $\mathcal C_{1,I'}$ of the Cayley graph of $G_n$ with generating set $A_0\sqcup\{t\}$. Let $n$ denote the diameter of $I'$. Consider $x,y\in\mathcal C_{1,I'}$ that are connected by a path $g$, seen as an element of $G_0=A_0*\langle t\rangle$. Because of the constraints to its image in $\Z$, we may write $g$ in the form
  \[g = t^{i_0}g_1^{t^{i_1}}\cdots g_k^{t^{i_k}}t^{i_{k+1}}\]
  where $-n\le i_0,i_{k+1}\le n$ and $0\le i_1,\dots,i_k\le n$. Then in $G_n$ we have $g= t^{i_0}g' t^{i_{k+1}}$ with $g'\in A_n$, a finite group. Therefore the number of vertices of $\mathcal C_{1,I'}$ reachable from $x$ is at most $4n^2\#A_n$.
\end{proof}

\subsection{Marked groups}\label{ss:markedgroups}
We consider groups generated by a fixed finite set $S$. These appear in the literature as \emph{marked groups}, or quotients of a fixed free group $F_S$. We denote by $B_G(r)$ the ball of radius $r$ in a marked group $G$. For example, if $S=A_0\cup\{t,t^{-1}\}$, then by Lemma~\ref{lem:generate} all $G_n$ are marked groups.

The \emph{space of marked groups}, introduced by Grigorchuk and Gromov, consists in the set $\mathscr G_S$ of marked groups, topologized by declaring two marked groups $G,H$ to be close if they agree on a large ball, namely if $B_G(r)$ and $B_H(r)$ are isomorphic as $S$-labeled graphs for some large $r$. See~\cite{champetier:marked} for a good survey.

\begin{lemma}\label{lem:converge}
  The groups $G_n$ above converge, in the space of marked groups, to $G$.

  Furthermore, this convergence may be made as fast as desired: for any function $\Phi\colon\mathscr G_S\to\N$ the sequence $G_n$ may be chosen in such a manner that for all $n$ the balls of radius $\Phi(G_n)$ coincide in $G_n$ and $G_{n+1}$.
\end{lemma}
\begin{proof}
  Every $A_n$ is a subgroup of $G$. The normal form for HNN extensions implies that every product of $A_0$ and $t,t^{-1}$ involving at most $n/2$ copies of $t$ remains in $A_n$, so the ball of radius $n/2$ remains constant from $G_n$ on.

  Let us return to the construction, to prove the quantitative version. We start with the $S$-marked group $G_0=A_0*\langle t\rangle$, with $S=A_0\sqcup\{t\}$. Then, inductively, we have
  \[G_n=A_n*_{A_{n-1}} \cong (A_n*_{A_{n-1}}A_n)*_{A_n},\]
  since in the latter group both of its $A_n$'s are identified by the HNN extension and there remains the identification of their $A_{n-1}$'s as in the former group.

  To construct $G_{n+1}$, we quotient $A_n*_{A_{n-1}}A_n$ above into a finite group $A_{n+1}$. For now, both of these groups are generated by $A_0\sqcup A_0$. Residual finiteness of free products with amalgation implies that there are sequences of such $A_{n+1}$ that converge to $A_n*_{A_{n-1}}A_n$ in $\mathscr G_{A_0\sqcup A_0}$. Now the normal form~\eqref{eq:hnn} implies the convergence of $G_{n+1}$ to $G_n$, so they can be made to agree on a large ball of radius $\Phi(G_n)$. In essence, this is continuity of the HNN extension construction in $\mathscr G_S$.  
\end{proof}

\section{Tree automata}
We briefly recall here some basics about infinite tree languages; see~\cite[Part~II]{Thomas90} for a good introduction.

Fix once and for all a finite set $S$, and a regular prefix-closed subset $W$ of $S^*$; here ``regular'' means that there is a finite state automaton accepting $W$. We naturally view $W$ as a tree, with an edge between $w$ and $w s$ for all $w\in W,s\in S$ such that $w s\in W$.

\begin{definition}[Tree language]
  Consider a finite alphabet $\Sigma$; then a \emph{tree language} is a subset $\mathfrak L$ of $\Sigma^W$, namely consists of $\Sigma$-vertex labelings of the tree $W$.
\end{definition}
It may psychologically be preferable to replace $W$ by the full tree $S^*$, and to view as tree language a subset of $\Sigma^{S^*}$; this may be achieved by replacing $\Sigma$ by $\Sigma\sqcup\{\mathvisiblespace\}$ and extending $x\in\Sigma^W$ to $x\colon S^*\to\Sigma\sqcup\{\mathvisiblespace\}$ by giving it the value $\mathvisiblespace$ outside $W$.

\begin{definition}[Rabin automaton]
  A \emph{Rabin automaton} $\mathscr A$ consists of a finite set $Q$, the \emph{stateset}, a \emph{start state} $q_0\in Q$, \emph{productions} $\Pi\subseteq Q\times\Sigma\times(Q\cup\{\dagger\})^S$, where a production is written $(q,\sigma)\to(q_s)_{s\in S'\subseteq S}$ keeping on the right all coordinates $\neq\dagger$, and a collection $(N_1,P_1),\dots,(N_k,P_k)$ of pairs of subsets of $Q$ called \emph{negative} and \emph{positive acceptance constraints}.

  A \emph{successful run} of $\mathscr A$ is a $Q\times\Sigma$-vertex labeling $r$ of $W$ such that the first ($Q$-) coordinate $x(1)_1$ of the root is $q_0$, at every $w\in W\subset S^*$ with descendants $w s_1,\dots,w s_k\in W$ the labels satisfy $(x(w),\{s\mapsto w(x s)_1\text{ or }\dagger\})\in\Pi$, and for every infinite ray $\xi\in S^\omega$ in $W$ there is an acceptance constraint $(P_i,N_i)$ such that the set of states $r(\xi_1\ldots\xi_n)_1$ that appear infinitely many times has trivial intersection with $N_i$ and non-trivial intersection with $P_i$.

  A language $\mathfrak L$ is \emph{regular} if it is the set of projections to $\Sigma^W$ of all successful runs of a Rabin automaton.
\end{definition}

The assumption that $W$ is a regular, prefix-closed language is equivalent to the one-element unary language over $W$, with $\Sigma=\{\cdot\}$, being regular: simply take for $Q$ the stateset of an automaton recognizing $W$, and as productions all $(q,\cdot,\{s\mapsto q\cdot s\text{ or }\dagger\})$.

The main properties of regular tree languages are that (1) they form an effective Boolean algebra: there are algorithms that, from Rabin automata defining languages $\mathfrak L,\mathfrak M$, produce Rabin automata defining $\mathfrak L^\complement,\mathfrak L\cap\mathfrak M,\mathfrak L\cup\mathfrak M$; and (2) emptiness is decidable: there is an algorithm that, given a Rabin automaton, decides whether it accepts at least one tree.

The proof of decidability of monadic second-order logic on regular trees, and therefore on virtually free groups, in fact proceeds as follows. Firstly, every monadic second-order statement can be converted into a Rabin automaton, in such a manner that the associated language is non-empty if and only if the statement is true; then emptiness of regular tree languages is shown to be decidable.

We include, in brief form, a proof of the following fundamental result, since we shall need a small extra property that can be extracted from the proof. Recall that a \emph{finite state automaton with output $\Sigma$}, for an output alphabet $\Sigma$, is a collection of finite state automata $(\mathscr A_\sigma)_{\sigma\in\Sigma}$ accepting disjoint languages; such a device defines a partial map $x\colon S^*\dashrightarrow\Sigma$ by $x(w)=\sigma$ whenever $\mathscr A_\sigma$ accepts $w$:
\begin{theorem}[Rabin, see~{\cite[Theorem~9.3]{Thomas90}}]\label{thm:rabin}
  If $\mathfrak L$ is a non-empty regular tree language, then $\mathfrak L$ contains a \emph{regular} tree, namely a labeling $x\colon W\to\Sigma$ that is given by a classical finite state automaton with alphabet $S$ and output $\Sigma$.
\end{theorem}
\begin{proof}[Sketch of proof]
  Let $\mathscr A$ be a Rabin automaton recognizing $\mathfrak L$. First, without loss of generality we may assume that $q_0$ never occurs in right sides of productions, and $\Sigma=\{\cdot\}$; this last step is achieved by replacing $Q$ with $Q\times\Sigma$, and the regular tree we are seeking is now a regular run $x\colon W\to Q$.
  
  A state $q\in Q$ is called \emph{live} unless $q=q_0$ or there is a single transition $(q,(q,\dots,q))$.

  Let $r\colon W\to Q$ be a successful run. The proof proceeds by induction on the number of live states in $r$. If there are none, then $r$ is constant and we are done. If some live state is missing from $r$, then $r$ will be accepted by an automaton with fewer live states, namely $\mathscr A$ with $q$ missing. Similarly, if there is a vertex $w$ such that $r(w)=q$ is live but some live state $q'$ does not appear beyond $w$ in $r$, then there are partial runs $r_1,r_2$ defined by ``running till state $q$ is reached'' on the modification of $\mathscr A$ in which $q$ is made non-live (by only allowing $(q,(q,\dots,q))$ as transitions out of $q$), respectively ``starting at state $q$ and never reaching $q'$''  on the automaton in which $q$ is made initial and $q'$ is removed from $\mathscr A$; by induction these partial runs can be assumed to be regular, and they can be stitched back into a regular run for $\mathscr A$.

  The main case is that in which all live states appear in $r$ beyond any vertex. We \textbf{choose a path} $\xi_0$ in $W$ such that all live states appear in
  \[Q_\infty=\{\text{those states that appear infinitely often on }\xi\}.\] Since $r$ is successful, there is an acceptance constraint, say $(N_1,P_1)$, such that $Q_\infty\cap N_1=\emptyset\neq Q_\infty\cap P_1$. Choose $q\in Q_\infty\cap P_1$, and again find two regular partial runs $r_1,r_2$, with $r_1$ as above and $r_2$ defined by ``starting at state $q$ and making a copy of $q$ non-live for further visits''. Again $r_1$ and $\omega$ copies of $r_2$ can be stitched together at all occurrences of state $q$. To check that the resulting run is successful, the interesting case is on paths $\xi$ on which $q$ appears infinitely often. For such paths, $(N_1,P_1)$ is a valid acceptance constraint: firstly $q\in P_1$ so $P_1$ intersects non-trivially the recurrent states. If some $q'\in N_1$ occurred infinitely often on $\xi$, then $q'$ would be non-live and therefore the only state on $\xi$, contradicting that $q$ is live and occurs infinitely often.    
\end{proof}

\section{Fooling snakes}
Below, we consider groups generated by a fixed finite set $S$. Recall the $S$-marked groups from Section~\ref{ss:markedgroups}.

\begin{definition}[Reach]
  Let $G$ be an $S$-marked group, let $A, B$ two subsets of $G$, and let $\Theta$ be a tileset. We say that \emph{snakes can reach $B$ from $A$} if there is a finite $\Theta$-snake with left endpoint in $A$ and right endpoint in $B$. If $\pi\colon G \to K$ is an epimorphism, then for $a, b \in K$ we say that \emph{snakes can $(G, K)$-reach $b$ from $a$} if it they can reach $\pi^{-1}(b)$ from $\pi^{-1}(a)$. Again this makes sense for all types of snakes, though we only need it for strong directed snakes.
\end{definition}

The next two propositions are the fundamental properties of snakes that we shall use to prove Theorem~\ref{thm:main}.
\begin{proposition}\label{prop:Lift}
  Let $G \to H \to K$ be epimorphisms of $S$-marked groups, suppose that $\Theta = (S', C, D)$ is a directed tileset for some finite $S'\subset F_S$, and consider $a,b\in K$. If a strong directed $\Theta$-snake can $(H, K)$-reach $b$ from $a$, then a strong directed $\Theta$-snake can $(G, K)$-reach $b$ from $a$.
\end{proposition}
\begin{proof}
  Consider a $\Theta$-snake in $H$ reaching a preimage of $b$ from a preimage of $a$. Lift this snake from $H$ to $G$ by following the directions in the tiles. The injective path is lifted to an injective path, and there are fewer adjacencies to be checked for validity. The endpoints still map to the same places in $K$, since $H$ is an intermediate quotient.
\end{proof}

\begin{definition}[Periodic snake]
  Let $\omega\colon\Z\to G$ be a snake. It is \emph{$p$-periodic}, for $p\ge1$, if there is an element $g\in G$, perforce of infinite order, such that $\omega(i+p)=\omega(i)g$ for all $i\in\Z$. We call $\omega$ \emph{periodic} if it is $p$-periodic for some $p\ge1$.

  Similarly, if $\Theta$ is a tileset, then a \emph{periodic $\Theta$-snake} is a snake obeying the tiling constraints $\Theta$, which is periodic in the above sense, and whose tiles are also periodic.
\end{definition}

\begin{proposition}\label{prop:PeriodicSnake}
  Let $G$ be a virtually free group, and suppose $\Theta = (S', C, D)$ is a directed tileset for some finite $S'\subset F_S$. If $G$ admits an infinite strong directed $\Theta$-snake, then it admits a periodic strong directed infinite snake.

  Furthermore, if $\Z$ is a quotient of $G$, and strong directed $(G,\Z)$-snakes can reach $m$ from $0$ for $|m|$ arbitrarily large, then there is a periodic snake with infinite projection to $\Z$.
\end{proposition}
\begin{proof}
  We begin with $G$, and exploit the assumption that it is virtually free. Firstly, there exists a regular normal form $W\subset S^*$ representing it, namely mapping bijectively and quasi-isometrically to $G$ via evaluation; and this regular language $W$ may effectively be computed from the coding of $G$. Furthermore, $W$ may be assumed to be prefix-closed, and representing a geodesic tree in $G$'s Cayley graph. Indeed, consider a finite-index free subgroup $F_X\le G$, and a transversal $T$ so that $G=F_X\cdot T$; we may then choose $W=\{x_1\cdots x_n t:x_i\in X\cup X^{-1},x_i\neq x_{i+1}^{-1},t\in T\}$.

  The \emph{directed snake language} is the tree language $\mathfrak S$ on $W$ with alphabet $\Sigma=S\cup\{\cdot\}$, in which vertex label $s\in S$ represents an arrow pointing in direction $s$ and $\cdot$ represents no arrow, and the Cayley graph of $G$, when decorated by these arrows, contains an infinite ray starting at $1$ and no other arrow.
  \begin{lemma}
    The language $\mathfrak S$ is regular.
  \end{lemma}
  \begin{proof}
    This directly follows from directed snakes being expressible in monadic second-order logic, see~\cite[Theorem~5]{AB23}.
  \end{proof}
  Note however that $\mathfrak S$ is \emph{not} a closed language, in the topological sense.
  
  We next consider the tiling constraints $\Theta$. They only require a finite amount of extra information at every vertex to check that they are enforced in its neighbourhood (this is where we require $W$ to be quasi-isometric to $G$); so the tree language $\mathfrak S_\Theta$ of $\Theta$-directed snakes is also regular.

    We are now ready to prove the proposition; this is in disguise a pumping argument applied to $\mathfrak S_{\Theta,\Delta}$, based on a Rabin's fundamental result. By Theorem~\ref{thm:rabin}, if there is an infinite snake then $\mathfrak S_\Theta$ is non-empty so there exists a regular tree $x\colon W\to\Sigma$.  In particular, there are $a\ne b\in W$ such $x(a)=x(b)$ is an arrow, namely is part of the snake, $b$ is in the subtree at $a$, and the automaton is in the same state at $a$ and at $b$. Write $g=a^{-1}b$. By symmetry, we may assume that $a$ is visited before $b$ along the snake. The labeling $x$ is defined by a deterministic finite state automaton with output, so the snake's movement is deterministic. Therefero, as we trace the path from $a$ to $b=a g$ along the snake, the same path is traced from $b$ to $b g$, etc.; so $x$ defines an ultimately periodic snake, and repeating infinitely many times the segment of $x$ between $a$ and $b$, translated by powers of $g$, gives the desired periodic snake.

  Now consider the \textbf{boldface} statement in the proof of Theorem~\ref{thm:rabin}. If there is a snake with unbounded projection to $\Z$, then there is a run $r$ of $\mathfrak S_\Theta$ with the property that it has live states at arbitrarily large $\Z$-positions; indeed if $r(w)$ is non-live then the whole subtree below $W$ is constant, and in particular cannot accommodate a snake. Therefore the path $\xi_0$ may be chosen in such a manner that its projection to $\Z$ is unbounded, and since copies of $r_2$ are stitched along $\pi_0$ the resulting regular tree will also carry a snake with infinite $\Z$-projection. All that is needed is that the segment of $\xi_0$ that repeats be chosen long enough that not only it contains all live states, but also it have two occurrences of $q$ that have distinct projections to $\Z$.
\end{proof}

\begin{proof}[Proof avoiding Rabin's Theorem~\ref{thm:rabin}]
  There is some magic involved in our use of Rabin's theorem; in particular, the proof proceeds by removing non-live states, which has the effect of passing from the quasi-tree $G$ to an actual tree such as the free group $F_X$ contained in it. It may be worthwhile to make explicit the argument, and avoid invoking such a (useful but dangerous) hammer.

  We begin as in the proof above with the tree language $\mathfrak S_\Theta$ of infinite snakes on a normal form for $G$.

  For simplicity, we may convert $\mathfrak S_\Theta$ to a tree language over the finite-index free subgroup $F_X$ of $G$, since the Cayley graph of $F_X$ is already a tree. For this, we keep track at each vertex $w\in F_X$ of what the snake does at each of the element of $w T$, where $T$ is the chosen transversal of $F_X$ in $G$. We may also assume that snakes move one $X$-step at a time, by interpolating them appropriately.

  Let us fix a tree in $\mathfrak S_\Theta$. There is a unique boundary point in $F_X$ to which the snake eventually goes, since it visits every node finitely many times. Let $\xi$ be the geodesic in $T_X$ going towards this boundary point. On $\xi$, by pigeonholing we may find two vertices $a<b$ around which everything ``looks the same'' locally, namely the number, order and directions of incoming and outgoing snakes are the same.

  We may then copy the tree labeling between $a$ and $b$ infinitely, by forcing the tree issued from $b$ to take the labeling of the tree issued from $a$. It is then clear that the snake moves infinitely far in this eventually periodic branch: a snake starts at the origin, and since locally we never end a snake (and copying trees does not break the rules in $\Theta$), that snake is going to be infinite. It has to go in the infinite branch, because the ``support'' of the configuration is at a bounded distance from that branch. This is because in the original configuration, the snake eventually stays in the $b$-branch, so it visits only finitely many nodes outside the b branch, and that finite part just gets copied around the periodic geodesic we constructed.

  Now, let us analyze the movement and show that the snake moves along an eventually periodic sequence of moves along the geodesic. Consider the sequence of nodes $a_1=a,a_2=b,a_3=\text{the $b$-node inside the tree copied from $a$ to $b$}$, etc. These are on a geodesic in the tree, and we know the snake visits all of them (but not necessarily the same way it visits $a$ and $b$, of course). Now, in each of them we have the same slots. In some of these slots, the snake enters from ``home'' and exits ``away'', and in some it enters from ``away'' and exits ``home'' (it might move also within slots internally, but these can be considered as ``no-move''). Here ``home'' is the direction of the origin, and ``away'' goes towards the ends of $F_X$. There must be an odd number of slots that get used, with the first one being entered from ``home'' and exited ``away''. The part between $a$ and $b$ in the original tree determines how the ``away''-outgoing slots of $a_i$ connect to the ``home''-incoming slots of $a_{i+1}$, and how the ``home''-outgoing slots connect to ``away''-incoming slots of $a_i$.

  Now, we may look at the movement of the original snake. In each $a_i$, there is some order in which it uses the ``home''-incoming slots and the ``away''-incoming slots. We can find $a_i$ and $a_{i+k}$ such that this sequence is exactly the same. Then an obvious induction shows that it remains the same between $a_i$ and $a_{i+nk}$ of all $n\in\N$, since the part between $a_{i+(n-1)k}$ and $a_{i+nk}$ is the same as that between $a_i$ and $a_{i+k}$. The number of slots that have been used in $a_{i+nk}$ is at least as large as the number of slots that have been used in $a_{i+(n+1)k}$: no more slots may be used on these two, so the snake will never return to their left; and after a bounded number of steps the same slots will be filled in $a_{i+(n+1)k}$ and now the snake is confined even further to the right.

  Suppose that there are $m$ slots in total at each node. Consider the nodes $a_i,a_{i+k},\dots,a_{i+mk}$, and the time when the snake first enters one of the slots of $a_{i+mk}$ (from ``home'' of course). Count how many slots have been used in each. As we argued this is nondecreasing as we go home. Therefore, we have some $a_{i+nk}$ and $a_{i+(n+1)k}$ where these numbers are the same. Now an easy induction shows that no more slots than that ever get used.

  In this version of the proof, assume also that the snake has unbounded projection to $\Z$. Then the vertices $a,b$ may be chosen above in such a manner that they have different projection to $\Z$, and then the snake will not only be quasi-geodesic, its projection to $\Z$ will also be quasi-geodesic.
\end{proof}

\begin{corollary}\label{cor:WhichCase}
  Given a tileset $\Theta$ and a virtually free group $G\in\mathcal F$, by some effective encoding, it is decidable which of the following holds:
  \begin{enumerate}
  \item There is an infinite $\Theta$-snake in $G$ with unbounded projection to $\Z$;
  \item There are infinite $\Theta$-snakes in $G$, but they all have bounded projection to $\Z$;
  \item There are no infinite $\Theta$-snakes in $G$.
  \end{enumerate}
\end{corollary}
\begin{proof}
  The existence of an infinite $\Theta$-snake can be coded in monadic second-order logic, by~\cite[Theorem~5]{AB23}, so it is decidable whether we are in the first two cases or the third.

  Now if there are infinite $\Theta$-snakes with bounded but arbitrarily large projection to $\Z$, then by the last part of Proposition~\ref{prop:PeriodicSnake} there is an infinite $\Theta$-snakes with unbounded projection to $\Z$.

  Therefore, there only remains the first case or the refined second case ``there is a bound $R$ such that all infinite $\Theta$-snakes have projection to $\Z$ of width at most $R$''.

  The first case can be semi-decided, by Proposition~\ref{prop:PeriodicSnake}, since periodic snakes can be effectively enumerated.

  The second case can also be semi-decided: for every $R\in\N$, the statement ``there is an infinite $\Theta$-snake whose projection to $\Z$ has range $> R$'' is expressible in monadic second-order logic. Indeed let $L_R\subset S^*$ be the language of words whose evaluation and projection to $\Z$ is not in $[-R,R]$. This language is regular, an automaton recognizing it having states $\{-R,1-R,\dots,R\}\cup\{\surd\}$, initial state $0$, accepting state $\surd$, and transitions that keep track of the $\Z$-projection, or accept if it falls out of the range $[-R,R]$. To test the existence of an infinite $\Theta$-snake whose projection to $\Z$ has width $>R$, we simply overlay on the snake predicate the condition that the snake's labels has a factor in $L_R$.
\end{proof}

We now state our main technical result. Consider the space $\mathcal M_S$ of $S$-marked groups; fix a generator $t\in S$, and an epimorphism $\pi\colon F_S\twoheadrightarrow\Z$ with $\pi(t)=1$. Let $\mathcal F$ be the subspace of virtually free groups $G\in\mathcal{M}$ through which $\pi$ factors; so groups $G\in\mathcal F$ come equipped with a homomorphism $G\twoheadrightarrow\Z$ still written $\pi$. Elements of $\mathcal F$ can be coded in an effective manner, for example by the coset table of finite-index free subgroup. We identify $\mathcal F$ with its set of codings, and henceforth can speak of computable properties of $\mathcal F$.

\begin{theorem}\label{thm:technical}
  There exists a computable function $\Phi\colon\mathcal{F}^* \to \N$ such that the following holds. Let $G \in \mathcal M$ be a group which is a limit of a computable sequence of virtually free groups $G_n \in \mathcal F$, such that
  \begin{enumerate}
  \item if $I \subset \Z$ is a finite interval and $r\in\N$ then, for large enough $n$, the $B(G_n,r)$-connected components of $\pi^{-1}(I)$ are finite;\label{technical:1}
  \item $G_{n+1}$ agrees with $G_n$ in the $\Phi((G_1, \dots, G_n))$-ball for each $n$.\label{technical:2}
  \end{enumerate}
  Then $G$ has decidable snake problem and infinite snake problem.
\end{theorem}

This statement is quite a mouthful, but intuitively it simply states that whenever we have a computable construction of limits of virtually free groups with a fixed generating set and satisfying the two technical properties, and we can at each step pick how large a ball we preserve at the next step, then we can produce groups $G$ with decidable snake and infinite snake problems by always preserving a large enough ball, where $\Phi$ gives the size bounds. Such a general construction was precisely described in the previous section (the HNN extensions $G_n$ are all generated by $S=A_0\cup\{t,t^{-1}\}$, by Lemma~\ref{lem:generate}; HNN extensions always come equipped with a homomorphism onto $\Z$; Assumption~\ref{technical:1} is given by Lemma~\ref{lem:locfinite} while Assumption~\ref{technical:2} follows from Lemma~\ref{lem:converge}.

One can think of it as follows: consider a game where Alice and Bob construct a group as a limit of virtually free groups in $\mathcal M$, all the while preserving a homomorphism onto $\Z$ taking value $1$ on $t$. Alice at each turn picks a number $r$, and Bob picks a proper quotient so that the ball of radius $r$ is preserved. Alice wins if at some point Bob is not able to construct a virtually free proper quotient, or if the group obtained in the limit actually has decidable snake problems, or if the first item fails and from some point on Bob no longer keeps making the $\Z$-fibers more coarsely disconnected (in the sense of that item). The function $f$ is a winning strategy for this game. Thus we obtain computable examples of groups with decidable snake problems by following any computable strategy for Bob such that a proper quotient is always picked.

We state Theorem~\ref{thm:technical} in this manner, with the expectation that slightly different constructions of the groups $G_n$ will lead to more interesting examples of non-virtually-free groups with decidable tiling properties.

\begin{proof}
Instead of describing the function $\Phi$, it is easier to explain how to play the game against Bob, to construct a sequence $(G_n)_n$, as this allows us to ``remember'' what we have in mind for our choices, rather than have to deduce what we have done in the past from a given sequence $(G_1, \ldots, G_n)$. So at each step we have access to a sequence of groups $G_1, \ldots, G_n$, and some finite mount of information (things we should keep track of in future steps that might give lower bounds on future radii). We then choose how big a ball $G_{n+1}$ should preserve (in addition to the items in the statement of the theorem).

  First of all, we enumerate all directed tile sets $(S_n,C_n,D_n)$ where $S_n\subset F_S$ is a finite memory set, $C_n \subseteq C'_n \times S_n$ is a finite colour set, and $D_n \subseteq S_n\times C_n \times S_n$. Initially, Bob gives us some group $G_0$ which is virtually free and has the two properties. Assume the groups $G_0, \ldots, G_n$ are fixed and let us explain how to pick the size of the ball to preserve in the next step.
  
  Let us consider the directed tile set $\Theta_n = (S_n,C_n,D_n)$ on the group $G_n$. Note that we can easily solve the snake problem, infinite snake problem and ouroboros problem on $G_n$, so when thinking about $G_{n+1}$, we can assume we know these solutions. We first consider only the snake problem (the infinite snake problem will be similar but easier, and we can pick the maximum of the radii required by both). We have to figure out a suitable size of preserved balls so that the snake problem is decidable for $G$ no matter how the following groups $G_m$ are picked. At stage $n$, we only explicitly think about the tile set $\Theta_n$, but we set up some future constraints for later stages (which will also just be about preserving large enough balls), so we don't have to come back to it.

   Suppose first that the snake problem has negative solution for $\Theta_n$ in $G_n$. Then there is a radius $r$ that witnesses this, meaning that there is no snake from $1$ to $G_n \setminus B(G_n,r)$ for some $r\in\N$. This choice of $r$, if also imposed at all future steps, guarantees that $\Theta_n$ will also not snake tile the limit group $G$.
  
  Suppose next that there is a snake for $\Theta_n$ in $G_n$. If it is an ouroboros, then there is a radius $r$ that witnesses it and again pick at least radius $r$ from now on.
  
  Suppose next that there is an infinite snake. Remember that for all $m$ we have epimorphisms $G_m\twoheadrightarrow G\twoheadrightarrow\Z$. Suppose first that snakes can reach arbitrarily large numbers in $\Z$. By Proposition~\ref{prop:PeriodicSnake}, there is an infinite directed strong snake $\omega$ with period $p$ which reaches such numbers. Now as long as $\omega(\{0,\dots,p\})$ and its $S_n$-neighbourhood are mapped injectively in future quotients, the snake survives all the way to $G$, since no power of $t$ is ever killed.
   
   Suppose next that snakes may only reach bounded values in $\Z$. Then we will keep track of this snake in the following steps, not yet knowing what the fate of its snake problem is. By Assumption~\eqref{technical:1}, the snake will not be able to move arbitrarily far, since eventually in $G_m$ the set of elements that may be reached while remaining at bounded distance in $\Z$. (For this, note that by Proposition~\ref{prop:Lift}, the $\Z$-distance we can move can never increase when passing to quotients.) So at that stage, if there is an ouroboros, then we start keeping it alive in the remaining construction, and if there isn't an ouroboros, then there is a bound on snake sizes and again it suffices to take a large enough radius to witness the non-tileability of $\Theta_n$ in every $G_m$.

   Finally, by Corollary~\ref{cor:WhichCase} we know in which case of the above alternatives we are.
   
   We move now to the infinite snake problem and show how the argument above should be adapted. If there is an infinite snake reaching arbitrarily large values in $\Z$, then again then there is a periodic snake by the first part of Proposition~\ref{prop:PeriodicSnake}, and we can simply keep it alive.
   
   If there is no infinite snake, or if all snakes have a bounded image in $\Z$, then a large enough radius guarantees that there will be no infinite snake in the limit $G$.
\end{proof}

\begin{remark}
  The above construction cannot extend to the ouroboros problem: on the free group $F$, it is easy to construct a directed tile set which does not admit an ouroboros, but admits one on \emph{every} quotient: on every edge of $F$'s Cayley graph put an arrow pointing towards $1$, and accept as ouroboroi all directed paths.

  Thus, one would have to understand which quotients introduce an ouroboros, and this seems similar to the issues involved in trying to show that the groups constructed here may have decidable domino problem.
\end{remark}

\bibliographystyle{plain}
\bibliography{bib}{}

\begin{thebibliography}{10}

\bibitem{AB23}
Nathalie Aubrun and Nicolas Bitar.
\newblock Domino snake problems on groups.
\newblock In {\em Fundamentals of computation theory}, volume 14292 of {\em
  Lecture Notes in Comput. Sci.}, pages 46--59. Springer, Cham, [2023]
  \copyright 2023.

\bibitem{BS18}
Alexis Ballier and Maya Stein.
\newblock The domino problem on groups of polynomial growth.
\newblock {\em Groups Geom. Dyn.}, 12(1):93--105, 2018.

\bibitem{BaTr78}
Benjamin Baumslag and Marvin Tretkoff.
\newblock Residually finite hnn extensions.
\newblock {\em Communications in Algebra}, 6(2):179--194, 1978.

\bibitem{Ba63}
Gilbert Baumslag.
\newblock On the residual finiteness of generalised free products of nilpotent
  groups.
\newblock {\em Transactions of the American Mathematical Society},
  106(2):193--209, 1963.

\bibitem{Berger66}
Robert Berger.
\newblock The undecidability of the domino problem.
\newblock {\em Mem. Amer. Math. Soc.}, 66:72, 1966.

\bibitem{Britton63}
John~L. Britton.
\newblock The word problem.
\newblock {\em Ann. of Math. (2)}, 77:16--32, 1963.

\bibitem{champetier:marked}
Christophe Champetier.
\newblock L'espace des groupes de type fini.
\newblock {\em Topology}, 39(4):657--680, 2000.

\bibitem{Et91}
Yael Etzion.
\newblock {\em On the solvability of domino snake problems}.
\newblock {MSc}. thesis, Dept. of Computer Science and Applied Mathematics, The
  Weizmann Institute of Science, Rehovot, Israel, 1991.

\bibitem{EtHaMy94}
Yael Etzion-Petruschka, David Harel, and Dale Myers.
\newblock On the solvability of domino snake problems.
\newblock {\em Theoretical Computer Science}, 131(2):243--269, 1994.

\bibitem{Ka94}
Jarkko Kari.
\newblock Reversibility and surjectivity problems of cellular automata.
\newblock {\em Journal of Computer and System Sciences}, 48(1):149--182, 1994.

\bibitem{Kari03}
Jarkko Kari.
\newblock Infinite snake tiling problems.
\newblock In {\em Developments in language theory}, volume 2450 of {\em Lecture
  Notes in Comput. Sci.}, pages 67--77. Springer, Berlin, 2003.

\bibitem{LS77}
Roger~C. Lyndon and Paul~E. Schupp.
\newblock {\em Combinatorial group theory}, volume Band 89 of {\em Ergebnisse
  der Mathematik und ihrer Grenzgebiete [Results in Mathematics and Related
  Areas]}.
\newblock Springer-Verlag, Berlin-New York, 1977.

\bibitem{My79}
Dale Myers.
\newblock Decidability of the tiling connectivity problem. abstract 79t-e42.
\newblock {\em Notices Amer. Math. Soc.}, 26(5):A--441, 1979.

\bibitem{Thomas90}
Wolfgang Thomas.
\newblock Automata on infinite objects.
\newblock In {\em Handbook of theoretical computer science, {V}ol.\ {B}}, pages
  133--191. Elsevier, Amsterdam, 1990.

\bibitem{Tr73}
Marvin Tretkoff.
\newblock The residual finiteness of certain amalgamated free products.
\newblock {\em Mathematische Zeitschrift}, 132:179--182, 1973.

\end{thebibliography}

\end{document}